\documentclass[psamsfonts,fceqn,leqno]{amsart}
\usepackage{mathrsfs,latexsym,amsfonts,amssymb,curves,epic}
\usepackage{color}

\newtheorem{theorem}{Theorem}[section]

\newtheorem{proposition}[theorem]{Proposition}

\newtheorem{question}[theorem]{Question}

\theoremstyle{definition}
\newtheorem{definition}[theorem]{Definition}
\newtheorem{remark}[theorem]{Remark}
\newtheorem{example}[theorem]{Example}

\begin{document}
\title[A note on complementary knowledge spaces]
{A note on complementary knowledge spaces}

\author{Fucai Lin}
\address{(Fucai Lin): School of mathematics and statistics,
 Minnan Normal University, Zhangzhou 363000, P. R. China}
\email{linfucai2008@aliyun.com; linfucai@mnnu.edu.cn}

\thanks{This work is supported by the Key Program of the Natural Science Foundation of Fujian Province (No: 2020J02043).}

  \keywords{knowledge structure, knowledge space, complementary knowledge space, discrete knowledge space}%insert keywords
  \subjclass[2000]{}%insert subject class

  %\date{\today}
  \begin{abstract}
The pair $(Q, \mathscr{K})$ is a {\it knowledge space} if $\bigcup\mathscr{K}=Q$ and $\mathscr{K}$ is closed under union, where $Q$ is a nonempty set and $\mathscr{K}$ is a family of subsets of $Q$. A knowledge space $(Q, \mathscr{K})$ is called {\it complementary} if there exists a non-discrete knowledge space $(Q, \mathscr{L})$ such that the following (i) and (ii) satisfy:

\smallskip
(i) for any $q\in Q$, there are finitely many $K_{1}, \cdots, K_{n}\in \mathscr{K}$ and $L_{1}, \cdots, L_{m}\in \mathscr{L}$ such that $$(\bigcap_{i=1}^{n}K_{i})\cap (\bigcap_{j=1}^{m}L_{j})=\{q\};$$

\smallskip
(ii) $\mathscr{K}\cap \mathscr{L}=\{\emptyset, Q\}$.

In this paper, the existence of a complementary knowledge space for each knowledge space is proved, and a method of the construction of complementary finite knowledge spaces is given.
  \end{abstract}

 \maketitle
\section{Introduction and Preliminaries}
In this paper, we always say that $Q$ is a field of {\it knowledge} is a non-empty set of questions or items, and a subset of $Q$ is said to be a {\it knowledge state} if an individual can master it under ideal conditions. Moreover, if a collection $\mathscr{H}$ of knowledge state satisfies $\{\emptyset, Q\}\subseteq\mathscr{H}$, then $\mathscr{H}$vis called {\it knowledge structure}, which is denoted by $(Q, \mathscr{H})$.
If the field can be omitted without ambiguity, sometimes we simply say that $\mathscr{H}$ is the knowledge structure.
Further, we say that a knowledge structure $\mathscr{H}$ is called a {\it knowledge
space} \cite{doignon2011knowledge}, which is an important type of knowledge structures.

Indeed, the theory of knowledge spaces (KST) was introduced by Doignon and Falmagne in 1999. Now it is a mathematical framework for the assessment of knowledge and advices for further learning, see \cite{doignon1985spaces,falmagne2011learning}. Based on individuals' responses to items, KST makes a dynamic evaluation process  \cite{doignon1985spaces}. In \cite{2009Danilov}, Danilov said that the concept of a knowledge space is a generalization of the concept of topological space. Then the authors in \cite{LWCL} systematically discussed the language of knowledge spaces in
pre-topology.

The knowledge structure in the domain of items is always obtained either by experts to determine the difficult and easy relationship between items, or by inference system. Different experts may obtain different knowledge structures in the same domain of items. If two knowledge structures in the same domain of items intersect as $\{\emptyset, Q\}$, then one explanation is that these two types of experts have a big difference in the understanding of the difficulty of items. For example, teachers can be considered as experts, and students can obtain knowledge structure by different teachers according to different teaching strategies. On the other hand, if we can combine different two types of experts of the understanding of the difficulty of items, then we may find that for each item $q$ there exist finitely many states of this two types of experts such that $q$ is the unique common item which we must master in order to reach these finitely many states. Therefore, we study the concept of complementary of two knowledge spaces and give a method to constructing a complementary knowledge space for each finite knowledge space.

Assume $\mathscr{F}\subseteq 2^{Q}$; then put $\mathscr{F}_{t}=\{K\in\mathscr{F}: t\in Q\}$ and $$t^{\ast}=\{r| \mathscr{H}_{r}=\mathscr{H}_{t}, r\in Q\}$$for each $t\in Q$. Then each $t^{\ast}$ is said to be a {\it notion} \cite{doignon2011knowledge}. Clearly, we have $$t^{\ast}=\{r\in Q| r^{\ast}=t^{\ast}\}.$$ Moreover, we say that $(Q, \mathscr{H})$ is {\it discriminative} if $t^{\ast}$ is single item for each $t\in Q$.

In \cite{XGLJ}, X. Ge and J. Li introduced the concept of $T_{0}$-knowledge structure and $T_{1}$-knowledge structure, which respectively satisfy $\mathscr{K}_{p}\neq \mathscr{K}_{q}$, and $\mathscr{K}_{p}\nsubseteq \mathscr{K}_{q}$, $\mathscr{K}_{q}\nsubseteq\mathscr{K}_{p}$ for any distinct items $p$ and $q$ in $Q$. Clearly, a knowledge structure is a discriminative knowledge structure iff it is a $T_{0}$-knowledge structure.

\begin{definition}
Assume $(Q, \mathscr{K})$ is a knowledge space. We say that

\smallskip
(1) $(Q, \mathscr{K})$ is {\it discrete} if, for each $q\in Q$, there exist finitely many $K_{1}, \cdots, K_{n}\in \mathscr{K}$ such that $\{q\}=\bigcap_{i=1}^{n}K_{i}$, and $(Q, \mathscr{K})$ is non-trivial whenever it is non-discrete and $\mathscr{K}\neq\{\emptyset, Q\}$;

\smallskip
(2) $(Q, \mathscr{K})$ is {\it complementary} if there exists a non-discrete knowledge space $(Q, \mathscr{L})$ such that the following two conditions (i) and (ii) satisfy:

\smallskip
(i) for each $q\in Q$, there are finitely many $K_{1}, \cdots, K_{n}\in \mathscr{K}$ and $L_{1}, \cdots, L_{m}\in \mathscr{L}$ so that $$(\bigcap_{i=1}^{n}K_{i})\cap (\bigcap_{j=1}^{m}L_{j})=\{q\};$$

\smallskip
(ii) $\mathscr{L}\cap \mathscr{K}=\{\emptyset, Q\}$.

\smallskip
Moreover, we say that $(Q, \mathscr{K})$ is {\it discriminative complementary} if $(Q, \mathscr{L})$ in (2) is also discriminative.
\end{definition}

\begin{remark}
Clearly, each discrete knowledge space is discriminative, but that the converse does not hold in general. Indeed, let $(Q, \mathscr{K})$ be a knowledge space, where $\mathscr{K}=\{\emptyset, Q, \{b, c\}, \{a, b\}\}$ on $Q=\{a, b, c\}$. Then $(Q, \mathscr{K})$ is discriminative; however, $(Q, \mathscr{K})$ is not discrete since $\bigcap\{K\in\mathscr{K}: c\in K\}=\{b, c\}$.
\end{remark}

\begin{definition}
A collection $(Q, \mathscr{K})$ is a {\it closure space} if $Q\in \mathscr{K}$ and $\bigcap\mathscr{K}^{\prime}\in\mathscr{K}$ for any non-empty subfamily $\mathscr{K}^{\prime}$ of $\mathscr{K}$.
\end{definition}

We denote $\mathcal{K}$ by the family of all knowledge spaces on $Q$. Then $\mathcal{K}$, ordered by inclusion, is a lattice (with infimum equal to intersection, and supremum obtained by taking all unions of a state in one knowledge space and a state in the other knowledge space). Clearly, the lattice $\mathcal{K}$ is canonically isomorphic to the lattice of closure spaces on $Q$.

\begin{definition}
Assume that $(Q, \mathscr{K})$ is a knowledge structure, and assume that $O$ is a
nonempty proper subset of $Q$. Then we say that the following family
$$\mathscr{K}_{|O}=\{K\cap O: K\in\mathscr{K}\}$$
is a {\it projection} of $\mathscr{K}$ on $O$.
\end{definition}

 \maketitle
\section{main results}
In this section, the existence of a complementary knowledge space for each knowledge space is proved in Theorem~\ref{ttttt}, and a method to constructing a complementary knowledge space for each finite knowledge space is given.

\begin{theorem}\label{ttttt}
Each non-trivial knowledge space $(Q, \mathscr{K})$ is complementary.
\end{theorem}

\begin{proof}
Suppose $(Q, \mathscr{K})$ is a non-trivial knowledge space. Let $\tau$ be the topology on $Q$ which is generated by $\mathscr{K}$ as a subbase. Then $\tau$ is not discrete since $\tau$ is non-trivial. By \cite{SA1966}, there is a topology $\sigma$ on $Q$ such that $\sigma$ is a complementary with $\tau$. Clearly, $(Q, \sigma)$ is a knowledge space. The proof is completed.
\end{proof}

Of course, there exists no discriminative complementary knowledge space, see the following example.

\begin{example}
There exists a finite discriminative knowledge space $(Q, \mathscr{K})$ which has no discriminative complementary knowledge space.
\end{example}

\begin{proof}
Let $Q=\{a, b, c\}$ and $\mathscr{K}=\{\emptyset, Q, \{b, c\}, \{a, c\}, \{a, b\}\}$. Clearly, $(Q, \mathscr{K})$ is a discriminative knowledge space. However, any non-trivial knowledge space on $Q$ is not discriminative complementary with $(Q, \mathscr{K})$. If not, let $(Q, \mathscr{L})$ is discriminative complementary with $(Q, \mathscr{K})$. Then $|L|\neq 2$ for any $L\in \mathscr{L}$. Since $(Q, \mathscr{L})$ is non-trivial, it follows that at most one of $\{a\}\in \mathscr{L}$, $\{b\}\in \mathscr{L}$ and $\{c\}\in \mathscr{L}$ holds. Without loss of generality, we may assume that $\{a\}\in \mathscr{L}$, then $\mathscr{L}=\{\emptyset, \{a\}, Q\}$, which is not discriminative, a contradiction.
\end{proof}

\begin{proposition}
A finite knowledge space $(Q, \mathscr{K})$ is discrete iff $\mathscr{K}$ is a $T_{1}$-knowledge structure.
\end{proposition}

\begin{proof}
Suppose that $(Q, \mathscr{K})$ is discrete. Fix any $q\in Q$. Since $(Q, \mathscr{K})$ is discrete, there exist $K_{1}, \cdots, K_{n}\in \mathscr{K}$ such that $\{q\}=\bigcap_{i=1}^{n}K_{i}$. For any $p\in Q\setminus\{q\}$, since $\{q\}=\bigcap_{i=1}^{n}K_{i}$, it follows that $\{q\}=\bigcap\mathscr{K}_{q}$, hence there exists a $K\in\mathscr{K}$ such that $q\in K$ and $p\not\in K$. Therefore, $\mathscr{K}$ is a $T_{1}$-knowledge structure.

Assume that $\mathscr{K}$ is a $T_{1}$-knowledge structure. Fix any $q\in Q$. For any $p\in Q\setminus\{q\}$, since $\mathscr{K}$ is a $T_{1}$-knowledge structure, there exists $K\in\mathscr{K}$ such that $q\in K$ and $p\not\in K$. Hence $\{q\}=\bigcap\mathscr{K}_{q}$. Since $Q$ is finite, $(Q, \mathscr{K})$ is discrete.
\end{proof}

\begin{example}
There exists an infinite $T_{1}$-knowledge space $(Q, \mathscr{K})$; however, it is not discrete.
\end{example}

\begin{proof}
Indeed, let $Q=\mathbb{R}$ and $\mathscr{K}=\{K: |Q\setminus K|\leq\omega, K\subset Q\}\cup\{\emptyset\}$. Then $(Q, \mathscr{K})$ is a $T_{1}$-knowledge space. However, for any finitely many $K_{1}, \cdots, K_{n}\in \mathscr{K}$, the intersection $\bigcap_{i=1}^{n}K_{i}$ is an uncountable set.
\end{proof}

\begin{question}
Assume that $(Q, \mathscr{K})$ is a non-trivial finite knowledge space. How many non-trivial knowledge spaces which are complementary with $\mathscr{K}$?
\end{question}

Suppose $(Q, \mathscr{K})$ is a non-trivial finite knowledge space in the rest of this paper. Then we give a method to construct a complementary knowledge space $(Q, \mathscr{L})$ for $(Q, \mathscr{K})$.

By \cite[Theorem 48]{LWCL} and a sketch of algorithm in \cite{LWCL}, we can take an arbitrary subset $D_{0}=\{q_{1}, \cdots, q_{m}\}$ of $(Q, \mathscr{K})$ such that the following (1)-(3) hold:

\smallskip
(1) If $\bigcap(\mathscr{K}\setminus\{\emptyset\})\neq\emptyset$, then $m=1$ and $D_{0}\subset\bigcap(\mathscr{K}\setminus\{\emptyset\})$; otherwise, for any $1< n\leq m$, there is $K\in\mathscr{K}$ such that $q_{n}\in K$ and $K\cap\{q_{1}, \cdots, q_{n-1}\}=\emptyset$.

\smallskip
(2) $K\cap D_{0}\neq\emptyset$ for any $K\in \mathscr{K}\setminus\{\emptyset\}$.

\smallskip
(3) If $m\neq1$, then there is an $i\leq n$ such that $\{q_{i}\}\not\in \mathscr{K}_{|D_{0}}$ (this is possible since $(Q, \mathscr{K})$ is non-trivial).

We consider $D_{0}$ together with the natural well-ordering, $q_{i}<q_{j}$ iff $i<j$. For the convenience, we add a new point $\infty\in D_{0}$ to the set $D_{0}$ and conclude that $q_{i}<\infty$ for each $i\leq m$. By our construction, we have $[d, \infty)=\{x\in D_{0}: d\leq x\}\in \mathscr{K}_{|D_{0}}$ for each $d\in D_{0}$. For each $i\leq m-1$, put $$q_{i}^{\prime}=\min\{a\in D_{0}\cup\{\infty\}: j>i,  [q_{i}, a)\in \mathscr{K}_{|D_{0}}\}.$$ Clearly we have $\{q_{m}\}\in \mathscr{K}_{|D_{0}}$. Then there exists a minimal natural number $n<m$ such that $\{[a_{i}, a_{i+1}): i\leq n\}\cup\{q_{m}\}$ cover the set $D_{0}$, where $a_{1}=q_{1}$, $a_{i+1}=a_{i}^{\prime}$ for each $i\leq n$ and it can happen that $a_{n+1}=\infty$. Let $$\mathscr{O}=\{(\bigcup_{i\leq n}[a_{i}, f(i)))\cup\{q_{m}\}: f\in (D_{0})^{n}, a_{i}<f(i)\leq a_{i+1}\ \mbox{for each}\ i\leq n\}.$$ By (3), we have $|\mathscr{O}|>1$. Take any subset $\mathscr{O}^{\prime}$ of $\mathscr{O}$ such that $\mathscr{O}^{\prime}$ is a base for some knowledge space on $D_{0}$ and for any $O\in\mathscr{O}$ there exist finitely many $O_{1}, \cdots, O_{p}\in \mathscr{O}^{\prime}$ with $O=\bigcap_{i=1}^{p}O_{i}$, and denote this knowledge space by $(D_{0}, \mathscr{D}_{0})$. We claim that $(D_{0}, \mathscr{D}_{0})$ is complementary with $\mathscr{K}_{|D_{0}}$.

Indeed, since all the sets $[a_{i}, a_{i+1})$ belong to $\mathscr{K}_{|D_{0}}$ and none of the sets $[a_{i}, c))$ with $a_{i}<c<a_{i+1}$ belongs to $\mathscr{K}_{|D_{0}}$, it follows that $\mathscr{K}_{|D_{0}}\cap \mathscr{D}_{0}=\{\emptyset, D_{0}\}$. In order to prove the knowledge space generated by $\mathscr{K}_{|D_{0}}\cup \mathscr{D}_{0}$ is the discrete knowledge space, we fix an arbitrary point $d\in D_{0}$. If $d=d_{m}$, then $\{d\}=\{d_{m}\}\in \mathscr{K}_{|D_{0}}$, hence the proof is completed. Otherwise there exists exactly one $j\leq n$ such that $d\in [a_{j}, a_{j+1})$. If $d$ is maximal in $[a_{j}, a_{j+1})$, then $\{d\}=[a_{j}, a_{j+1})\cap [d, d_{m})$; since $[d, d_{m})\in \mathscr{K}_{|D_{0}}$, the proof is completed. Assume that $d$ is not maximal in $[a_{j}, a_{j+1})$. Set $$b=\min\{a\in [a_{j}, a_{j+1}): c<a\}.$$ Then there exists $O\in \mathscr{O}$ such that $O\cap [a_{j}, a_{j+1})=[a_{j}, b)$, hence there are finitely many $O_{1}, \cdots, O_{p}\in \mathscr{O}^{\prime}$ such that $O=\bigcap_{i=1}^{p}O_{i}$. Moreover, $[d, d_{m})$ and $[a_{j}, a_{j+1})$ belong to $\mathscr{K}_{|D_{0}}$. Therefore, $$[d, d_{m})\cap [a_{j}, a_{j+1})\cap O_{1}\cap\cdots\cap O_{p}=[d, a_{j+1})\cap [a_{j}, b)=\{d\},$$which completes the proof of the claim.

Since $Q$ is finite, we can define a finitely many subsets $\{D_{i}: 0\leq i\leq n_{0}\}$ of $Q$ with respective knowledge spaces $\mathscr{D}_{i}$ such that the following (a)-(c) hold.

\smallskip
(a) Each $(D_{i}, \mathscr{D}_{i})$ is complementary with $\mathscr{K}_{|D_{i}}$;

\smallskip
(b) For each $i\leq n_{0}$, we have $K\cap D_{i}\neq\emptyset$ for each $K\in \mathscr{K}_{|Q\setminus (\bigcup_{j<i}D_{j})}$, where $D_{-1}=\emptyset$;

\smallskip
(c) $Q=\bigcup_{i=0}^{n_{0}}D_{i}$.

\smallskip
Let $m_{0}\leq n_{0}$ be the minimal natural number such that $\bigcup_{i=0}^{m_{0}}D_{i}\in\mathscr{K}$. Then put $$X=Q\setminus\bigcup_{i=0}^{m_{0}}D_{i}$$ and $$\mathscr{B}=\{W\cup X:W\in\mathscr{D}_{0}\}\cup\bigcup_{i=1}^{n_{0}}\mathscr{D}_{i}.$$ Let $\mathscr{L}$ be the knowledge space which is generated by the family $\mathscr{B}$. Clearly, we have $\mathscr{D}_{i}\subset \mathscr{L}$ for each $0<i\leq n_{0}$. Next we prove the following two claims:

\smallskip
{\bf Claim 1:} $\mathscr{K}\cap\mathscr{L}=\{\emptyset, Q\}$.

\smallskip
Let $K\in \mathscr{K}\cap\mathscr{L}$ with $K\neq\emptyset$. Since $K\in \mathscr{K}$, it follows that $K\cap D_{0}\neq\emptyset$, then $D_{0}\subset K$ because $K\cap D_{0}\in \mathscr{K}_{|D_{0}}$ and $\mathscr{D}_{0}\cap\mathscr{K}_{|D_{0}}=\{\emptyset, D_{0}\}$. Similarly, we can prove that $D_{i}\subset K$ for each $i\leq n_{0}$. Therefore, $K=Q$.

\smallskip
{\bf Claim 2:} For each $q\in Q$, there exist $K_{1}, \cdots, K_{r}\in \mathscr{K}$ and $L_{1}, \cdots, L_{t}\in \mathscr{L}$ such that $$(\bigcap_{i=1}^{r}K_{i})\cap (\bigcap_{j=1}^{t}L_{j})=\{q\}.$$

\smallskip
Clearly, there exists exactly one $i_{0}\leq n_{0}$ such that $q\in D_{i_{0}}$. If $i_{0}=0$, then there exist $K_{1}^{\ast}, \cdots, K_{r}^{\ast}\in \mathscr{K}_{|D_{0}}$ and $L_{1}^{\ast}, \cdots, L_{t}^{\ast}\in \mathscr{D}_{0}$ such that $(\bigcap_{i=1}^{r}K_{i}^{\ast})\cap (\bigcap_{j=1}^{t}L_{j}^{\ast})=\{q\}$. For each $i\leq r$, there exists $K_{i}\in\mathscr{K}$ such that $K_{i}^{\ast}=K_{i}\cap D_{i_{0}}$. Then it is easily checked that $$(Q\setminus X)\cap(\bigcap_{i=1}^{r}K_{i})\cap (\bigcap_{j=1}^{t}(L_{j}^{\ast}\cup X))=\{q\}.$$ The proof is completed since $Q\setminus X, K_{i}\in \mathscr{K}$ for each $i\leq r$ and $L_{j}^{\ast}\cup X\in \mathscr{L}$ for each $i\leq t$..

Assume that $i_{0}>0$. Since $\mathscr{K}_{|D_{i_{0}}}$ is complementary with $\mathscr{D}_{i_{0}}$, there exist $K_{1}^{\prime}, \cdots, K_{r}^{\prime}\in \mathscr{K}_{|D_{i_{0}}}$ and $L_{1}, \cdots, L_{t}\in \mathscr{D}_{i_{0}}$ such that $(\bigcap_{i=1}^{r}K_{i}^{\prime})\cap (\bigcap_{j=1}^{t}L_{j})=\{q\}$. For any $i\leq r$, there is $K_{i}\in\mathscr{K}$ with $K_{i}^{\prime}=K_{i}\cap D_{i_{0}}$. Obviously, we have $$(\bigcap_{i=1}^{r}K_{i})\cap (\bigcap_{j=1}^{t}L_{j})=\{q\}.$$ Since $\mathscr{D}_{i_{0}}\subset \mathscr{L}$, it follows that $(Q,  \mathscr{L})$ is a complementary knowledge space with $(Q, \mathscr{K})$.

Now we give an example to show our method.

\begin{example}
Let $Q=\{a, b, c, d, e, f, g\}$ and $\mathscr{K}=\{\emptyset, Q, \{a, b, c, d, e, f\}, \{b, c, d, e, f\}, \{a, b, c, e, f\}, \\ \{a, b, c, d, f\},
\{a, b, c, d, e\}, \{a, b, d, f\}, \{b, c, e, f\}, \{b, c, d, f\}, \{b, c, d, e\}, \{a, b, c, e\}, \{a, b, c, d\}, \{b, d, f\}, \\ \{b, c, e\}, \{b, c, d\},
\{a, b, d\}, \{a, b, c\}, \{b, d\}, \{b, c\}, \{a, c\}, \{a, b\}, \{b\}, \{c\}\}$.

\smallskip
Clearly, $(Q, \mathscr{K})$ is a $T_{0}$-knowledge space. Let $D_{0}=\{q_{1}=e, q_{2}=b, q_{3}=c\}$. Then $D_{0}$ satisfies the conditions (1)-(3) above. It is easily checked that we have a knowledge space $(D_{0}, \mathscr{D}_{0})$, where $\mathscr{D}_{0}=\{\emptyset, \{e\}, \{b, e\}, \{b, c, e\}\}$. Then it is easy to check that $D_{1}=\{a, d, f, g\}$ with $$g_{1}=g<g_{2}=a<g_{3}=d< g_{4}=f$$ and a knowledge structure $\mathscr{D}_{1}=\{\emptyset, D_{1}, \{g\}, \{g, a\}, \{g, a, d\}\}$. Next let $$\mathscr{B}=\{Q, \{e, a, d, f, g\}, \{a, d, f, g\}, \{a, d, f, g, b, e\}, \{g, a, d\}, \{g, a\}, \{g\}\}.$$ It is easy to verify that the knowledge space $(Q, \mathscr{L})$ generated by $\mathscr{B}$ is complementary with $(Q, \mathscr{K})$.
\end{example}

\smallskip
{\bf Declarations}

\smallskip
{\bf Conflicts of Interest/Competing Interests} On behalf of all authors, the corresponding author states that there is no conflict of interest.
  %%%%%%%%%%%%%%%%%%%%%%%%%%%%

  \end{document}